\def\year{2019}\relax
\documentclass[letterpaper]{article} 
\usepackage{aaai19}  
\usepackage{times}  
\usepackage{helvet}  
\usepackage{courier}  
\usepackage{url}  
\usepackage{graphicx}  
\usepackage{amssymb,amsmath,amsthm}
\usepackage{algorithm}
\usepackage{algpseudocode}

\newtheorem{Assumption}{Assumption}
\newtheorem{Thm}{Theorem}

\newtheorem{Lem}{Lemma}
\newtheorem{Cor}{Corollary}
\newtheorem{Rem}{Remark}

\begin{document}
%
\title{Parallel Restarted SGD with Faster Convergence and Less Communication: Demystifying Why Model Averaging Works for Deep Learning}
\author{Hao~Yu, Sen~Yang, Shenghuo~Zhu\\
Machine Intelligence Technology Team, Alibaba Group (U.S.) Inc., Bellevue, WA.\\
}
\maketitle

\begin{abstract}
\noindent  In distributed training of deep neural networks, parallel mini-batch SGD is widely used to speed up the training process by using multiple workers. It uses multiple workers to sample local stochastic gradient in parallel, aggregates all gradients in a single server to obtain the average, and update each worker's local model using a SGD update with the averaged gradient.  Ideally, parallel mini-batch SGD can achieve a linear speed-up of the training time (with respect to the number of workers) compared with SGD over a single worker. However, such linear scalability in practice is significantly limited by the growing demand for gradient communication as more workers are involved.  Model averaging, which {\bf periodically} averages individual models trained over parallel workers, is another common practice used for distributed training of deep neural networks since \cite{Zinkevich10NIPS} \cite{McDonald10NACCL}.  Compared with parallel mini-batch SGD, the communication overhead of model averaging is significantly reduced. Impressively, tremendous experimental works have verified that model averaging can still achieve a good speed-up of the training time as long as the averaging interval is carefully controlled. However, it remains a mystery in theory why such a simple heuristic works so well. This paper provides a thorough and rigorous theoretical study on why model averaging can work as well as parallel mini-batch SGD with significantly less communication overhead.

\end{abstract}

\section{Introduction}

Consider the distributed training of deep neural networks over multiple workers  \cite{Dean12NIPS}, where all workers can access all or partial training data and aim to find a common model that yields the minimum training loss.  Such a scenario can be modeled as the following distributed parallel non-convex optimization
\begin{align}
\min_{ \mathbf{x}\in \mathbb{R}^{m}}  \quad f(\mathbf{x})  \overset{\Delta}{=} \frac{1}{N}\sum_{i=1}^N f_{i}(\mathbf{x}) \label{eq:problem-form}
\end{align}
where $N$ is the number of nodes/workers and each $ f_{i}(\mathbf{x}) \overset{\Delta}{=}\mathbb{E}_{\zeta_{i}\sim \mathcal{D}_{i}} [ F_{i}(\mathbf{x}; \zeta_{i})]$ is a smooth non-convex function where $\mathcal{D}_i$ can be possibly different for different $i$.  Following the standard  stochastic optimization setting, this paper assumes each worker can locally observe unbiased independent stochastic gradients (around the last iteration solution $\mathbf{x}_{i}^{t-1}$) given by $\mathbf{G}_{i}^{t} = \nabla F_{i}(\mathbf{x}^{t-1}_{i}; \zeta_{i}^{t})$ with $\mathbb{E}_{\zeta_{i}^{t}\sim \mathcal{D}_{i}}[\mathbf{G}_{i}^{t} \vert \boldsymbol{\zeta}^{[t-1]}] = \nabla f_{i}(\mathbf{x}^{t-1}_{i}), \forall i$ where $\boldsymbol{\zeta}^{[t-1]} \overset{\Delta}{=} [\zeta_{i}^{\tau}]_{i\in\{1,2,\ldots,N\}, \tau\in\{1,\ldots,t-1\}}$ denotes all the randomness up to iteration $t-1$.

One classical parallel method to solve problem \eqref{eq:problem-form} is to sample each worker's local stochastic gradient in parallel, aggregate all gradients in a single server to obtain the average, and update each worker's local solution using the averaged gradient in its SGD step\footnote{Equivalently, we can let the server update its solution using the averaged gradient and broadcast this solution to all local workers. Another equivalent implementation is to let each worker take a single SGD step using its own gradient and send the updated local solution to the server; let the server calculate the average of all workers' updated solutions and refresh each worker's local solution with the averaged version.} \cite{Dekel12JMLR} \cite{Li14NIPS}.  Such a classical method, called {\bf parallel mini-batch SGD} in this paper, is conceptually equivalent to a single node Stochastic Gradient Descent (SGD) with a batch size $N$ times large and achieves $O(1/\sqrt{NT})$ convergence with a linear speed-up with respect to (w.r.t.) the number of workers \cite{Dekel12JMLR}.  Since every iteration of parallel mini-batch SGD requires exchanging of local gradient information among all workers, the corresponding communication cost is quite heavy and often becomes the performance bottleneck. 

There have been many attempts to reduce communication overhead in parallel mini-batch SGD.  One notable method called {\bf decentralized parallel SGD (D-PSGD)} is studied in \cite{Lian17NIPS}\cite{Jiang17NIPS} \cite{Lian18ICML}. Remarkably, D-PSGD can achieve the same $O(1/\sqrt{NT})$ convergence rate as parallel mini-batch SGD, i.e., the  linear speed-up w.r.t. the number of workers is preserved, without requiring a single server to collect stochastic gradient information from local workers. However, since D-PSGD requires each worker to exchange their local solutions/gradients with its neighbors at {\bf every} iteration, the total number of communication rounds in D-PSGD is the same as that in parallel mini-batch SGD.  Another notable method to reduce communication overhead in parallel mini-batch SGD is to let each worker use compressed gradients rather than raw gradients for communication. For example, {\bf quantized SGD} studied in \cite{Seide14Interspeech}\cite{Alistarh17NIPS}\cite{Wen17NIPS} or {\bf sparsified SGD} studied in \cite{Strom15Interspeech}\cite{Dryden16MLHPC}\cite{Aji17EMNLP} allow each worker to pass low bit quantized or sparsified gradients to the server at {\bf every} iteration by sacrificing the convergence to a mild extent. Similarly to D-PSGD, such gradient compression based methods require message passing at every iteration and hence their total number of communication rounds is still the same as that in parallel mini-batch SGD. 

Recall that parallel mini-batch SGD can be equivalently interpreted as a procedure where at each iteration each local worker first takes a single SGD step and then replaces its own solution by the average of individual solutions.  With a motivation to reduce the number of inter-node communication rounds, a lot of works suggest to reduce the frequency of averaging individual solutions in parallel mini-batch SGD.  Such a method is known as  {\bf model averaging} and has been widely used in practical training of deep neural networks. Model averaging can at least date back to \cite{Zinkevich10NIPS} \cite{McDonald10NACCL} where  individual models are averaged only at the last iteration before which all workers simply run SGD in parallel.  The method in \cite{Zinkevich10NIPS} \cite{McDonald10NACCL}, referred to as {\bf one-shot averaging}, uses only one single communication step at the end and is numerically shown to have good solution quality in many applications. However, it is unclear whether the one-shot averaging can preserve the linear speed-up w.r.t. the number of workers. In fact, \cite{Zhang16ArXiv} shows that one-shot averaging can yield inaccurate solutions for certain non-convex optimization. As a remedy, \cite{Zhang16ArXiv} suggests more frequent averaging should be used to improve the performance.  However,  the understanding on how averaging frequency can affect the performance of parallel SGD is quite limited in the current literature.  Work \cite{Zhou17ArXiv} proves that by averaging local worker solutions only every $I$ iterations, parallel SGD has convergence rate $O(\sqrt{I}/\sqrt{NT})$ for non-convex optimization.\footnote{In this paper, we shall show that if $I$ is chosen as $I = O(T^{1/4}/N^{3/4})$, parallel SGD for non-convex optimization does not lose any factor in its convergence rate.} That is, the convergence slows down by a factor of $I$ by saving $I$ times inter-node communication. A recent exciting result reported in \cite{Stich18ArXiv} proves that for strongly-convex minimization, model averaging can achieve a linear speed-up w.r.t. $N$ as long as the averaging (communication) step is performed once at least every $I = O(\sqrt{T}/\sqrt{N})$ iterations. Work \cite{Stich18ArXiv} provides the first theoretical analysis that demonstrates the possibility of achieving the same linear speedup attained by parallel mini-batch SGD with strictly less communication for {\bf strongly-convex} stochastic optimization.  However, it remains as an open question in \cite{Stich18ArXiv} whether it is possible to achieve $O(1/\sqrt{NT})$ convergence for non-convex optimization, which is the case of deep learning.  

On the other hand, many experimental works \cite{Povey15ICLR} \cite{Chen16ICASSP} \cite{McMahan17AISTATS} \cite{Su18ArXiv} \cite{Kamp18ArXiv} \cite{Lin18ArXiv} observe that model averaging can achieve a superior performance for various deep learning applications. One may be curious whether these positive experimental results are merely coincidences for special case examples or can be attained universally. In this paper, we shall show that model averaging indeed can achieve $O(1/\sqrt{NT})$ convergence for non-convex optimization by averaging only every $I = O(T^{1/4}/N^{3/4})$ iterations. That is, the same $O(1/\sqrt{NT})$ convergence is preserved for non-convex optimization while communication overhead is saved by a factor of $O(T^{1/4}/N^{3/4})$. To our knowledge, this paper is the first\footnote{After the preprint \cite{Yu18ArXiv} of this paper is posted on ArXiv in July 2018, another work \cite{WangJoshi18ArXiv} subsequently analyzes the convergence rate of model averaging for non-convex optimization. Their independent analysis relaxes our bounded second moment assumption but further assumes all $f_i(\mathbf{x})$ in formulation \eqref{eq:problem-form} are identical, i.e, all workers must access a common training set when training deep neural networks.} to present provable convergence rate guarantees (with the linear speed-up w.r.t. number of workers and less communication) of model averaging for non-convex optimization such as deep learning and provide guidelines on how often averaging is needed without losing the linear speed-up.

Besides reducing the communication cost, the method of model averaging also has the advantage of reducing privacy and security risks in the {\bf federated learning} scenario recently proposed by Google in \cite{McMahan17AISTATS}. This is because model averaging only passes deep learning models, which are shown to preserve good differential privacy, and does not  pass raw data or gradients owned by each individual worker.

\section{Parallel Restarted SGD and Its Performance Analysis}

Throughout this paper, we assume problem \eqref{eq:problem-form} satisfies the following assumption.

\begin{Assumption}\label{ass:basic}~
\begin{enumerate}
\item {\bf Smoothness: } Each function $f_{i}(\mathbf{x})$ is smooth with modulus $L$.
\item {\bf Bounded variances and second moments:} There exits constants $\sigma > 0$ and $G>0$ such that 
\begin{align*}
\mathbb{E}_{\zeta_{i}\sim \mathcal{D}_{i}} \Vert \nabla F_{i}(\mathbf{x}; \zeta_{i}) - \nabla f_{i}(\mathbf{x})\Vert^{2} \leq \sigma^{2}, \forall \mathbf{x}, \forall i
\end{align*}
\begin{align*}
\mathbb{E}_{\zeta_{i}\sim \mathcal{D}_{i}} \Vert \nabla F_{i}(\mathbf{x}; \zeta_{i}) \Vert^{2} \leq G^{2}, \forall \mathbf{x}, \forall i
\end{align*}
\end{enumerate}
\end{Assumption}

\begin{algorithm}
\caption{Parallel Restarted SGD (PR-SGD)}\label{alg:parallel-sgd}
\begin{algorithmic}[1]
\State {\bf Input:}  Initialize $\mathbf{x}_i^0 = \overline{\mathbf{y}} \in \mathbb{R}^m$. Set learning rate $\gamma > 0$ and node synchronization interval (integer) $I>0$ 
\For{$t=1$~\text{to}~$T$}
\State  Each node $i$ observes an unbiased stochastic gradient $\mathbf{G}_{i}^{t}$ of $f_{i}(\cdot)$ at point $\mathbf{x}_{i}^{t-1}$
\If {$t$\text{~is a multiple of~}$I$, i.e., $t~\text{mod}~I = 0$,} 
\State Calculate node average $\overline{\mathbf{y}} \overset{\Delta}{=}\frac{1}{N}\sum_{i=1}^{N} \mathbf{x}_{i}^{t-1}$
\State Each node $i$ in parallel updates its local solution 
\begin{align}
\mathbf{x}_{i}^{t} =\overline{\mathbf{y}} - \gamma \mathbf{G}_{i}^{t},\quad \forall i \label{eq:parallel-sgd-x-update-use-y}
\end{align}
\Else
\State Each node $i$ in parallel updates its local solution 
\begin{align}
\mathbf{x}_{i}^{t} = \mathbf{x}_{i}^{t-1} - \gamma \mathbf{G}_{i}^{t}, \quad \forall i \label{eq:parallel-sgd-x-update}
\end{align}
\EndIf
\EndFor
\end{algorithmic}
\end{algorithm}

Consider the simple parallel SGD described in Algorithm \ref{alg:parallel-sgd}.  If we divide iteration indices into epochs of length $I$, then in each epochs all $N$ workers are running SGD in parallel with the same initial point $\overline{\mathbf{y}}$ that is the average of final individual solutions from the previous epoch. This is why we call Algorithm \ref{alg:parallel-sgd} ``{\bf Parallel Restarted SGD}''.  The ``model averaging" technique used as a common  practice for training deep neural networks can be viewed as a special case since Algorithm \ref{alg:parallel-sgd} calculates the model average to obtain $\overline{\mathbf{y}}$ every $I$ iterations and performs local SGDs at each worker otherwise.  Such an algorithm is different from elastic averaging SGD (EASGD) proposed in \cite{Zhang15NIPS} which periodically drags each local solution towards their average using a controlled weight. Note that synchronization (of iterations) across $N$ workers is not necessary inside each epoch of Algorithm \ref{alg:parallel-sgd}.  Furthermore, inter-node communication is only needed to calculate the initial point at the beginning of each epoch and is longer needed inside each epoch. As a consequence, Algorithm \ref{alg:parallel-sgd} with $I>1$ reduces its number of communication rounds by a factor of $I$ when compared with the classical parallel mini-batch SGD.  The linear speed-up property (w.r.t. number of workers) with $I>1$ is recently proven only for strongly convex optimization in \cite{Stich18ArXiv}. However, there is no theoretical guarantee on whether the linear speed-up with $I>1$ can be preserved for non-convex optimization, which is the case of deep neural networks.

Fix iteration index $t$, we define 
\begin{align}
\overline{\mathbf{x}}^{t} \overset{\Delta}{=}  \frac{1}{N}\sum_{i=1}^{N} \mathbf{x}_{i}^{t} \label{eq:node-average-x}
\end{align}
as the average of local solution $\mathbf{x}_{i}^{t}$ over all $N$ nodes.  It is immediate that 
\begin{align}
\overline{\mathbf{x}}^{t} = \overline{\mathbf{x}}^{t-1} - \gamma \frac{1}{N}\sum_{i=1}^{N} \mathbf{G}_{i}^{t} \label{eq:overline-x-update}
\end{align}
Inspired by earlier works on distributed stochastic optimization \cite{Zhang12NIPS} \cite{Lian17NIPS} \cite{Mania17SIOPT} \cite{Stich18ArXiv} where convergence analysis is performed for an aggregated version of individual solutions, this paper focuses on the convergence rate analysis of $\overline{\mathbf{x}}^t$ defined in \eqref{eq:node-average-x}.  An interesting observation from \eqref{eq:overline-x-update} is: Workers in Algorithm \ref{alg:parallel-sgd} run their local SGD independently for most iterations,  however, they still jointly update their node average using a dynamic similar to SGD.The main issue in \eqref{eq:overline-x-update} is an ``inaccurate" stochastic gradient, which is a simple average of individual stochastic gradients at points different from $\overline{\mathbf{x}}^t$, is used. Since each worker in Algorithm \ref{alg:parallel-sgd} periodically restarts its SGD with the same initial point, deviations between each local solution $\mathbf{x}_{i}^{t}$ and $\overline{\mathbf{x}}^t$ are expected to be controlled by selecting a proper synchronization interval $I$. The following useful lemma relates quantity  $\mathbb{E} [\Vert \overline{\mathbf{x}}^{t} - \mathbf{x}^{t}_{i}\Vert^{2}]$ and algorithm parameter $I$. A similar lemma is proven in \cite{Stich18ArXiv}.

\begin{Lem} \label{lm:diff-avg-per-node}
Under Assumption \ref{ass:basic}, Algorithm \ref{alg:parallel-sgd} ensures
\begin{align*}
\mathbb{E} [\Vert \overline{\mathbf{x}}^{t} - \mathbf{x}^{t}_{i}\Vert^{2}] \leq 4\gamma^{2} I^{2} G^{2}, \forall i, \forall t
\end{align*}
where $\overline{\mathbf{x}}^{t}$ is defined in \eqref{eq:node-average-x} and $G$ is the constant defined in Assumption \ref{ass:basic}.
\end{Lem}
\begin{proof}
Fix $t \geq 1$ and $i\in\{1,2,\ldots, N\}$. Note that Algorithm \ref{alg:parallel-sgd} calculates the node average $\overline{\mathbf{y}} \overset{\Delta}{=} \frac{1}{N} \sum_{i=1}^{N} \mathbf{x}_{i}^{t-1}$ every $I$ iterations. Consider the largest $t_{0} \leq  t$ such that $\overline{\mathbf{y}}= \overline{\mathbf{x}}^{t_{0}}$ at iteration $t_{0}$ in Algorithm \ref{alg:parallel-sgd}. (Note that such $t_{0}$ must exist and $t - t_{0} \leq I$.)  We further note, from the update equations \eqref{eq:parallel-sgd-x-update-use-y} and \eqref{eq:parallel-sgd-x-update} in Algorithm \ref{alg:parallel-sgd}, that
\begin{align}
\mathbf{x}_{i}^{t}=  \overline{\mathbf{y}}  -\gamma \sum_{\tau = t_{0}+1}^{t}\mathbf{G}_{i}^{\tau} \label{eq:pf-diff-avg-per-node-eq1}
\end{align}
By \eqref{eq:overline-x-update}, we have
\begin{align*}
\overline{\mathbf{x}}^{t}  = \overline{\mathbf{y}} -\gamma \sum_{\tau = t_{0}+1}^{t} \frac{1}{N}\sum_{i=1}^{N}\mathbf{G}_{i}^{\tau}
\end{align*}

Thus, we have
{\small
\begin{align*}
&\mathbb{E} [\Vert \mathbf{x}^{t}_{i} - \overline{\mathbf{x}}^{t} \Vert^{2}] \\=&  \mathbb{E} [\Vert \gamma \sum_{\tau = t_{0}+1}^{t} \frac{1}{N}\sum_{i=1}^{N}\mathbf{G}_{i}^{\tau} -\gamma \sum_{\tau = t_{0}+1}^{t}\mathbf{G}_{i}^{\tau} \Vert^{2}]\\
=& \gamma^{2} \mathbb{E}  [\Vert \sum_{\tau = t_{0}+1}^{t} \frac{1}{N}\sum_{i=1}^{N}\mathbf{G}_{i}^{\tau} -\sum_{\tau = t_{0}+1}^{t}\mathbf{G}_{i}^{\tau}\Vert^{2}]\\
\overset{(a)}{\leq}& 2\gamma^{2} \mathbb{E} [\Vert \sum_{\tau = t_{0}+1}^{t} \frac{1}{N}\sum_{i=1}^{N}\mathbf{G}_{i}^{\tau}\Vert^{2} + \Vert \sum_{\tau = t_{0}+1}^{t}\mathbf{G}_{i}^{\tau}\Vert^{2} ]\\
\overset{(b)}{\leq}& 2\gamma^{2} (t-t_{0}) \mathbb{E} [ \sum_{\tau = t_{0}+1}^{t}\Vert  \frac{1}{N}\sum_{i=1}^{N}\mathbf{G}_{i}^{\tau}\Vert^{2} + \sum_{\tau = t_{0}+1}^{t}\Vert \mathbf{G}_{i}^{\tau}\Vert^{2}]\\
\overset{(c)}{\leq}& 2\gamma^{2} (t-t_{0}) \mathbb{E} [\sum_{\tau = t_{0}+1}^{t} (\frac{1}{N}\sum_{i=1}^{N}\Vert \mathbf{G}_{i}^{\tau}\Vert^{2}) +\sum_{\tau = t_{0}+1}^{t}\Vert \mathbf{G}_{i}^{\tau}\Vert^{2}]\\
 \overset{(d)}{\leq} & 4\gamma^{2} I^{2} G^{2}
\end{align*}
}%
where (a)-(c) follows by using the inequality $\Vert \sum_{i=1}^{n} \mathbf{z}_{i}\Vert^{2} \leq n \sum_{i=1}^{n} \Vert \mathbf{z}_{i}\Vert^{2}$ for any vectors $\mathbf{z}_{i}$ and any positive integer $n$ (using $n=2$ in (a), $n=t-t_0$ in (b) and $n=N$ in (c)); and (d) follows from Assumption \ref{ass:basic}.
\end{proof}

\begin{Thm} \label{thm:rate} 
Consider problem \eqref{eq:problem-form} under Assumption \ref{ass:basic}. If $0 < \gamma \leq \frac{1}{L}$ in Algorithm \ref{alg:parallel-sgd}, then for all $T\geq 1$, we have {\scriptsize $$\frac{1}{T} \sum_{t=1}^{T} \mathbb{E}[\Vert \nabla f(\overline{\mathbf{x}}^{t-1})\Vert^{2}] \leq \frac{2}{\gamma T} (f(\overline{\mathbf{x}}^{0}) -f^{\ast}) + 4 \gamma^{2} I^{2} G^{2} L^2 + \frac{L}{N}\gamma \sigma^{2}$$}%
where $f^{\ast}$ is the minimum value of problem \eqref{eq:problem-form}. 
\end{Thm}
\begin{proof}
Fix $t\geq 1$. By the smoothness of $f$ , we have
{\small 
\begin{align}
\mathbb{E}[f(\overline{\mathbf{x}}^{t})]\leq &\mathbb{E}[f(\overline{\mathbf{x}}^{t-1})] +  \mathbb{E}[\langle \nabla f(\overline{\mathbf{x}}^{t-1}), \overline{\mathbf{x}}^{t} - \overline{\mathbf{x}}^{t-1}\rangle]  \nonumber \\&~+ \frac{L}{2} \mathbb{E}[\Vert \overline{\mathbf{x}}^{t} - \overline{\mathbf{x}}^{t-1}\Vert^{2}] \label{eq:pf-thm-rate-eq1}
\end{align}
}%

Note that 
{\footnotesize
\begin{align}
 &\mathbb{E}[\Vert \overline{\mathbf{x}}^{t} - \overline{\mathbf{x}}^{t-1}\Vert^{2}] \overset{(a)}{=} \gamma^{2} \mathbb{E}[ \Vert  \frac{1}{N}\sum_{i=1}^{N} \mathbf{G}_{i}^{t}\Vert^{2}] \nonumber\\ 
 \overset{(b)}{=}& \gamma^{2} \mathbb{E}[ \Vert  \frac{1}{N}\sum_{i=1}^{N} ( \mathbf{G}_{i}^{t} - \nabla f_{i}(\mathbf{x}_{i}^{t-1})) \Vert^{2}]  + \gamma^{2} \mathbb{E} [ \Vert \frac{1}{N} \sum_{i=1}^{N} \nabla f_{i}(\mathbf{x}_{i}^{t-1})\Vert^{2}] \nonumber\\
\overset{(c)}{=}& \gamma^{2} \frac{1}{N^{2}}\sum_{i=1}^{N} \mathbb{E}[\Vert  \mathbf{G}_{i}^{t} - \nabla f_{i}(\mathbf{x}_{i}^{t-1})\Vert^{2}] + \gamma^{2} \mathbb{E} [ \Vert \frac{1}{N} \sum_{i=1}^{N} \nabla f_{i}(\mathbf{x}_{i}^{t-1})\Vert^{2}] \nonumber\\
\overset{(d)}{\leq }& \frac{1}{N} \gamma^{2} \sigma^{2}  +  \gamma^{2} \mathbb{E} [ \Vert \frac{1}{N} \sum_{i=1}^{N} \nabla f_{i}(\mathbf{x}_{i}^{t-1})\Vert^{2}] \label{eq:pf-thm-rate-eq2}
\end{align}}%
where (a) follows from \eqref{eq:overline-x-update}; (b) follows by noting that $\mathbb{E}[\mathbf{G}_{i}^{t}] = \nabla f_{i}(\mathbf{x}_{i}^{t-1})$ and applying the basic inequality $\mathbb{E}[\Vert \mathbf{Z} \Vert^{2}] = \mathbb{E} [ \Vert \mathbf{\mathbf{Z}} - \mathbb{E}[\mathbf{Z}]\Vert^{2}] + \Vert\mathbb{E}[\mathbf{Z}] \Vert^{2}$ that holds for any random vector $\mathbf{Z}$; (c) follows because each $\mathbf{G}_{i}^{t} - \nabla f_{i}(\mathbf{x}_{i}^{t-1})$ has $\mathbf{0}$ mean and is independent across nodes; and (d) follows from Assumption \ref{ass:basic}. 

We further note that 
{\small
\begin{align}
&\mathbb{E}[\langle \nabla f(\overline{\mathbf{x}}^{t-1}), \overline{\mathbf{x}}^{t} - \overline{\mathbf{x}}^{t-1}\rangle] \nonumber\\
\overset{(a)}{=}& -\gamma \mathbb{E} [\langle \nabla f(\overline{\mathbf{x}}^{t-1}), \frac{1}{N} \sum_{i=1}^{N} \mathbf{G}_{i}^{t}\rangle] \nonumber \\
\overset{(b)}{=}& -\gamma \mathbb{E} [\langle \nabla f(\overline{\mathbf{x}}^{t-1}), \frac{1}{N} \sum_{i=1}^{N} \nabla f_{i} (\mathbf{x}_{i}^{t-1})\rangle ] \nonumber \\
\overset{(c)}{=}& - \frac{\gamma}{2} \mathbb{E} \big[ \Vert \nabla f(\overline{\mathbf{x}}^{t-1})\Vert^{2} + \Vert \frac{1}{N} \sum_{i=1}^{N} \nabla f_{i} (\mathbf{x}_{i}^{t-1})\Vert^{2}  \nonumber\\ &\qquad\quad- \Vert \nabla f(\overline{\mathbf{x}}^{t-1}) - \frac{1}{N} \sum_{i=1}^{N} \nabla f_{i} (\mathbf{x}_{i}^{t-1})\Vert^{2}\big] \label{eq:pf-thm-rate-eq3}
\end{align}}%
where (a) follows from \eqref{eq:overline-x-update}; (b) follows because 
{\footnotesize
	\begin{align*}
	&\mathbb{E}[\langle \nabla f(\overline{\mathbf{x}}^{t-1}), \frac{1}{N} \sum_{i=1}^{N} \mathbf{G}_{i}^{t}\rangle] \\
	=& \mathbb{E}[\mathbb{E}[\langle \nabla f(\overline{\mathbf{x}}^{t-1}), \frac{1}{N} \sum_{i=1}^{N} \mathbf{G}_{i}^{t}\rangle | \boldsymbol{\zeta}^{[t-1]}]] \\
	=& \mathbb{E}[\langle \nabla f(\overline{\mathbf{x}}^{t-1}), \frac{1}{N} \sum_{i=1}^{N} \mathbb{E}[\mathbf{G}_{i}^{t}| \boldsymbol{\zeta}^{[t-1]}]\rangle ]\\
	 =& \mathbb{E}[\langle \nabla f(\overline{\mathbf{x}}^{t-1}), \frac{1}{N} \sum_{i=1}^{N} \nabla f_{i}(\mathbf{x}_{i}^{t-1})\rangle ]
	\end{align*}
}%
where the first equality follows by the iterated law of expectations, the second equality follows because $\overline{\mathbf{x}}^{t-1}$ is determined by $\boldsymbol{\zeta}^{[t-1]}= [\boldsymbol{\zeta}^{1}, \ldots, \boldsymbol{\zeta}^{t-1}]$ and the third equality follows by $\mathbb{E}[\mathbf{G}_{i}^{t} | \boldsymbol{\zeta}^{[t-1]}] = \mathbb{E}[\nabla F_{i}(\mathbf{x}_{i}^{t-1};\zeta^{t}_{i}) | \boldsymbol{\zeta}^{[t-1]}] = \nabla f_{i}(\mathbf{x}_{i}^{t-1})$; and (c) follows from the basic identity $\langle \mathbf{z}_{1}, \mathbf{z}_{2}\rangle = \frac{1}{2} \big( \Vert \mathbf{z}_{1}\Vert^{2} + \Vert \mathbf{z}_{2}\Vert^{2} - \Vert \mathbf{z}_{1} - \mathbf{z}_{2}\Vert^{2} \big)$ for any two vectors $\mathbf{z}_{1}, \mathbf{z}_{2}$ of the same length.

Substituting \eqref{eq:pf-thm-rate-eq2} and \eqref{eq:pf-thm-rate-eq3} into \eqref{eq:pf-thm-rate-eq1} yields 
{\footnotesize \begin{align}
&\mathbb{E}[f(\overline{\mathbf{x}}^{t})] \nonumber\\
\leq &\mathbb{E}[f(\overline{\mathbf{x}}^{t-1})] - \frac{\gamma - \gamma^{2}L}{2} \mathbb{E} [\Vert \frac{1}{N} \sum_{i=1}^{N} \nabla f_{i} (\mathbf{x}_{i}^{t-1})\Vert^{2}] \nonumber \\
 &- \frac{\gamma}{2} \mathbb{E}[\Vert \nabla f(\overline{\mathbf{x}}^{t-1})\Vert^{2}]  \nonumber\\ &+ \frac{\gamma}{2}  \mathbb{E}[
\Vert \nabla f(\overline{\mathbf{x}}^{t-1}) - \frac{1}{N} \sum_{i=1}^{N} \nabla f_{i} (\mathbf{x}_{i}^{t-1}) \Vert^{2} ] + \frac{L}{2N} \gamma^{2} \sigma^{2} \nonumber\\
\overset{(a)}{\leq}& \mathbb{E}[f(\overline{\mathbf{x}}^{t-1})]  - \frac{\gamma - \gamma^{2}L}{2} \mathbb{E}[\Vert \frac{1}{N} \sum_{i=1}^{N} \nabla f_{i} (\mathbf{x}_{i}^{t-1})\Vert^{2}]  \nonumber\\ &- \frac{\gamma}{2} \mathbb{E} [\Vert \nabla f(\overline{\mathbf{x}}^{t-1})\Vert^{2}] + 2\gamma^{3} I^{2} G^{2} L^2 + \frac{L}{2N} \gamma^{2} \sigma^{2} \label{eq:pf-thm-rate-eq6}\\
\overset{(b)}{\leq} &\mathbb{E}[f(\overline{\mathbf{x}}^{t-1})]  - \frac{\gamma}{2} \mathbb{E}[\Vert \nabla f(\overline{\mathbf{x}}^{t-1})\Vert^{2}] + 2\gamma^{3} I^{2} G^{2} L^2 + \frac{L}{2N} \gamma^{2} \sigma^{2} \label{eq:pf-thm-rate-eq7}
\end{align}
}%
where (b) follows from $0 < \gamma \leq \frac{1}{L}$ and (a) follows because 
{\small \begin{align}
&\mathbb{E}[ \Vert \nabla f(\overline{\mathbf{x}}^{t-1}) - \frac{1}{N} \sum_{i=1}^{N} \nabla f_{i} (\mathbf{x}_{i}^{t-1})\Vert^{2}] \nonumber \\
 =& \mathbb{E} [ \Vert \frac{1}{N} \sum_{i=1}^{N}\nabla f_{i}(\overline{\mathbf{x}}^{t-1}) - \frac{1}{N} \sum_{i=1}^{N} \nabla f_{i} (\mathbf{x}_{i}^{t-1})\Vert^{2}] \nonumber \\
=& \frac{1}{N^{2}} \mathbb{E} [\Vert \sum_{i=1}^{N} \big( \nabla f_{i}(\overline{\mathbf{x}}^{t-1}) - \nabla f_{i} (\mathbf{x}_{i}^{t-1}) \big)\Vert^{2}] \nonumber \\
\leq& \frac{1}{N} \mathbb{E} [ \sum_{i=1}^{N} \Vert \nabla f_{i}(\overline{\mathbf{x}}^{t-1}) - \nabla f_{i} (\mathbf{x}_{i}^{t-1})\Vert ^{2}] \nonumber \\
\leq& L^2\frac{1}{N} \sum_{i=1}^{N}\mathbb{E}[ \Vert \overline{\mathbf{x}}^{t-1} - \mathbf{x}_{i}^{t-1}\Vert^{2}] \nonumber \\
\leq& 4 \gamma^{2} I^{2} G^{2} L^2 \nonumber
\end{align}
}%
where the first inequality follows by using $\Vert \sum_{i=1}^{N} \mathbf{z}_{i}\Vert^{2} \leq N \sum_{i=1}^{N} \Vert \mathbf{z}_{i}\Vert^{2}$ for any vectors $\mathbf{z}_{i}$; the second inequality follows from the smoothness of each $f_{i}$ by Assumption \ref{ass:basic}; and the third inequality follows from Lemma \ref{lm:diff-avg-per-node}.

Dividing \eqref{eq:pf-thm-rate-eq7} both sides by $\frac{\gamma}{2}$ and rearranging terms yields
{\small \begin{align}
&\mathbb{E}\left [\Vert \nabla f(\overline{\mathbf{x}}^{t-1})\Vert^{2}\right] \nonumber \\
\leq& \frac{2}{\gamma} \left(\mathbb{E}\left[f(\overline{\mathbf{x}}^{t-1})\right]  - \mathbb{E}\left[f(\overline{\mathbf{x}}^{t})\right]\right) +4 \gamma^{2} I^{2} G^{2} L^2 + \frac{L}{N}\gamma \sigma^{2} \label{eq:pf-thm-rate-eq8}
\end{align}
}%
Summing over $t\in\{1,\ldots, T\}$ and dividing both sides by $T$ yields
\begin{align*}
&\frac{1}{T} \sum_{t=1}^{T} \mathbb{E}\left [\Vert \nabla f(\overline{\mathbf{x}}^{t-1})\Vert^{2}\right] \\
\leq &\frac{2}{\gamma T} \left(f(\overline{\mathbf{x}}^{0}) - \mathbb{E}\left[f(\overline{\mathbf{x}}^{T})\right]\right) +  4\gamma^{2} I^{2} G^{2} L^2 + \frac{L}{N}\gamma \sigma^{2} \\
\overset{(a)}{\leq} &  \frac{2}{\gamma T} \left(f(\overline{\mathbf{x}}^{0}) -f^{\ast}\right) +  4\gamma^{2} I^{2} G^{2} L^2 + \frac{L}{N}\gamma \sigma^{2}
\end{align*}
where (a) follows because $f^{\ast}$ is the minimum value of problem \eqref{eq:problem-form}.
\end{proof}

The next corollary follows by substituting suitable $\gamma, I$ values into Theorem \ref{thm:rate}.
\begin{Cor}\label{cor:rate}
Consider problem \eqref{eq:problem-form} under Assumption \ref{ass:basic}. Let $T\geq N$. 
\begin{enumerate}
\item If we choose $\gamma =\frac{\sqrt{N}}{L\sqrt{T}}$ in Algorithm \ref{alg:parallel-sgd}, then we have $\frac{1}{T} \sum_{t=1}^{T} \mathbb{E}[\Vert \nabla f(\overline{\mathbf{x}}^{t-1})\Vert^{2}] \leq  \frac{2L}{\sqrt{NT}} \left(f(\overline{\mathbf{x}}^{0}) -f^{\ast}\right) +  \frac{4N}{T}I^{2}G^{2} + \frac{1}{\sqrt{NT}}\sigma^{2}$.
\item If we further choose $I \leq \frac{T^{1/4}}{N^{3/4}}$, then  $\frac{1}{T} \sum_{t=1}^{T} \mathbb{E}[\Vert \nabla f(\overline{\mathbf{x}}^{t-1})\Vert^{2}] \leq \frac{2L}{\sqrt{NT}} (f(\overline{\mathbf{x}}^{0}) -f^{\ast}) +  \frac{4}{\sqrt{NT}} G^{2} + \frac{1}{\sqrt{NT}}\sigma^{2} = O(\frac{1}{\sqrt{NT}}) $
where $f^{\ast}$ is the minimum value of problem \eqref{eq:problem-form}. 
\end{enumerate}
\end{Cor}

\begin{Rem}
For non-convex optimization, it is generally impossible to develop a convergence rate for objective values.  In Theorem \ref{thm:rate} and Corollary \ref{cor:rate}, we follow the convention in literature \cite{GhadimiLan13SIOPT} \cite{Lian17NIPS} \cite{Alistarh17NIPS} to use the (average) expected squared gradient norm to characterize the convergence rate. Note that the average can be attained in expectation by taking each $\overline{\mathbf{x}}^{t-1}$ with an equal probability $1/T$.
\end{Rem}

From Theorem \ref{thm:rate} and Corollary \ref{cor:rate}, we have the following important observations:

\begin{itemize}
\item {\bf Linear Speedup:}  By part (1) of  Corollary \ref{cor:rate}, Algorithm \ref{alg:parallel-sgd} with any fixed constant $I$ has convergence rate $O(\frac{1}{\sqrt{NT}} + \frac{N}{T})$. If $T$ is large enough, i.e., $T > N^{3}$, then the term $\frac{N}{T}$ is dominated by  the term $\frac{1}{\sqrt{NT}} $ and hence Algorithm \ref{alg:parallel-sgd} has convergence rate $O(\frac{1}{\sqrt{NT}})$. That is,  our algorithm achieves a linear speed-up with respect to the number of workers.  Such linear speedup for stochastic non-convex optimization was previously attained by decentralized-parallel stochastic gradient descent (D-PSGD) considered in \cite{Lian17NIPS} by requiring at least $T > N^{5}$. See, e.g.,  Corollary 2 in \cite{Lian17NIPS}.\footnote{In fact, for a ring network considered in Theorem 3 in \cite{Lian17NIPS}, D-PSGD requires a even larger $T$ satisfying $T > N^9$ since its implementation depends on the network topology. In contrast, the linear speedup of our algorithm is irrelevant to the network topology.}

\item {\bf Communication Reduction:} Note that Algorithm \ref{alg:parallel-sgd} requires inter-node communication only at the iterations that are multiples of $I$. By Corollary \ref{alg:parallel-sgd}, it suffices to choose any $I \leq \frac{T^{1/4}}{N^{3/4}}$ to ensure the $O(\frac{1}{\sqrt{NT}})$ convergence of our algorithm. That is, compared with parallel mini-batch SGD or the D-PSGD in \cite{Lian17NIPS}, the number of communication rounds in our algorithm can be reduced by a factor $\frac{T^{1/4}}{N^{3/4}}$.   Although Algorithm \ref{alg:parallel-sgd} does not describe how the node average $\overline{\mathbf{y}}$ is obtained at each node, in practice, the simplest way is to introduce a parameter server that collects all local solutions and broadcasts their average as in parallel mini-batch SGD \cite{Li14NIPS}. Alternatively, we can perform an {\bf all-reduce} operation on the local models(without introducing a server) such that all nodes obtain $\overline{\mathbf{y}}$ independently and simultaneously. (Using an all-reduce operation among all nodes to obtain gradients averages has been previously suggested in \cite{Goyal17ArXiv} for distributed training of deep learning.)


\end{itemize}

\section{Extensions}

\subsection{Using Time-Varying Learning Rates}
Note that Corollary \ref{cor:rate} assumes time horizon $T$ is known and uses a constant learning rate in Algorithm \ref{alg:parallel-sgd}.
In this subsection, we consider the scenario where the time horizon $T$ is not known beforehand and develop a variant of Algorithm \ref{alg:parallel-sgd} with time-varying rates to achieve the same computation and communication complexity. Compared with Algorithm \ref{alg:parallel-sgd}, Algorithm \ref{alg:parallel-sgd-time-varying} has the advantage that its accuracy is being improved automatically as it runs longer.

\begin{algorithm}
\caption{PR-SGD with Time-Varying Learning Rates }\label{alg:parallel-sgd-time-varying}
\begin{algorithmic}[1]
\State {\bf Input:}  Set time-varying epoch learning rates $\gamma^{s} > 0$. 
\State {\bf Initialize:} Initialize $\mathbf{x}_i^{0,K^{0}}  = \overline{\mathbf{x}}^{0} \in \mathbb{R}^m$.
\For{epoch index $s=1$~\text{to}~$S$}
\State  Set epoch length $K^s$ and initialize $\mathbf{x}_{i}^{s,0} = \frac{1}{N}\sum_{i=1}^{N} \mathbf{x}_{i}^{s-1,K^{s-1}}$ to be the node average of local worker solutions from the last epoch.
\For{$k=1$~\text{to}~$K^{s}$}
\State Each node $i$ observes an unbiased gradient $\mathbf{G}_{i}^{s,k}$ of $f_{i}(\cdot)$ at point $\mathbf{x}_{i}^{s,k-1}$ and in parallel updates
\begin{align}
\mathbf{x}_{i}^{s,k} =\mathbf{x}_{i}^{s,k-1} - \gamma^{s} \mathbf{G}_{i}^{s,k},\quad \forall i \label{eq:parallel-sgd-time-varying-x-update}
\end{align}
\EndFor
\EndFor
\end{algorithmic}
\end{algorithm}

Although Algorithm \ref{alg:parallel-sgd-time-varying} introduces the concept of epoch for the convenience of description,  we note that it is nothing but a parallel restarted SGD where each worker restarts itself every epoch using the node average of the last epoch's final solutions as the initial point. If we sequentially reindex $\{\mathbf{x}_{i}^{s,k}\}_{s\in\{1,\ldots,S\}, k\in\{1,\ldots, K^{s}\}}$ as $\mathbf{x}_{i}^{t}$ (note that all $\mathbf{x}_{i}^{s,0}$ are ignored since $\mathbf{x}_{i}^{s,0} = \mathbf{x}_{i}^{s-1,K^{s-1}}$), then Algorithm \ref{alg:parallel-sgd-time-varying} is mathematically equivalent to Algorithm \ref{alg:parallel-sgd} except that time-varying learning rates $\gamma^{s}$ are used in different epochs. Similarly to \eqref{eq:node-average-x}, we can define $\overline{\mathbf{x}}^{s,k}$ via $\overline{\mathbf{x}}^{s,k} \overset{\Delta}{=}  \frac{1}{N}\sum_{i=1}^{N} \mathbf{x}_{i}^{s,k}$ and have
{\small 
\begin{align}
\overline{\mathbf{x}}^{s,k} = \overline{\mathbf{x}}^{s,k-1} - \gamma \frac{1}{N}\sum_{i=1}^{N} \mathbf{G}_{i}^{s,k} 
\end{align}
}

\begin{Thm}\label{thm:rate-time-varying}
Consider problem \eqref{eq:problem-form} under Assumption \ref{ass:basic}. If we choose $K^{s} = \lceil \frac{s^{1/3}}{N}\rceil$ and $\gamma^s = \frac{N}{s^{2/3}}$ in Algorithm \ref{alg:parallel-sgd-time-varying}, then for all $S\geq1$, we have\footnote{A logarithm factor $\log(NT)$ is hidden in the notation $\widetilde{O}(\cdot)$.}
{\small 
\begin{align*}
\frac{1}{\sum_{s=1}^{S}\sum_{k=1}^{K^{s}}{\gamma^{s}}}  \sum_{s=1}^{S} \sum_{k=1}^{K^{s}} \mathbb{E}\left [ \gamma^s \Vert \nabla f(\overline{\mathbf{x}}^{s, k-1})\Vert^{2}\right] \leq \widetilde{O}(\frac{1}{\sqrt{NT}})
\end{align*}
}%
where $T = \sum_{s=1}^{S} K^{s}$.
\end{Thm}
\begin{proof}
See Supplement.
\end{proof}

\subsection{Asynchronous Implementations in Heterogeneous Networks}

Algorithm \ref{alg:parallel-sgd} requires all workers to compute the average of individual solutions every $I$ iterations and  synchronization among local workers are not needed before averaging. However, the fastest worker still needs to wait until all the other workers finish $I$ iterations of SGD even if it finishes its own $I$ iteration SGD much earlier. (See Figure \ref{fig:1} for a $2$ worker example where one worker is significantly faster than the other. Note that orange ``syn'' rectangles represent the procedures to compute the node average.) As a consequence, the computation capability of faster workers is wasted.  Such an issue can arise quite often in heterogeneous networks where nodes are equipped with different hardwares. 

\begin{figure}[h!] 
\centering
\includegraphics[width=0.48\textwidth]{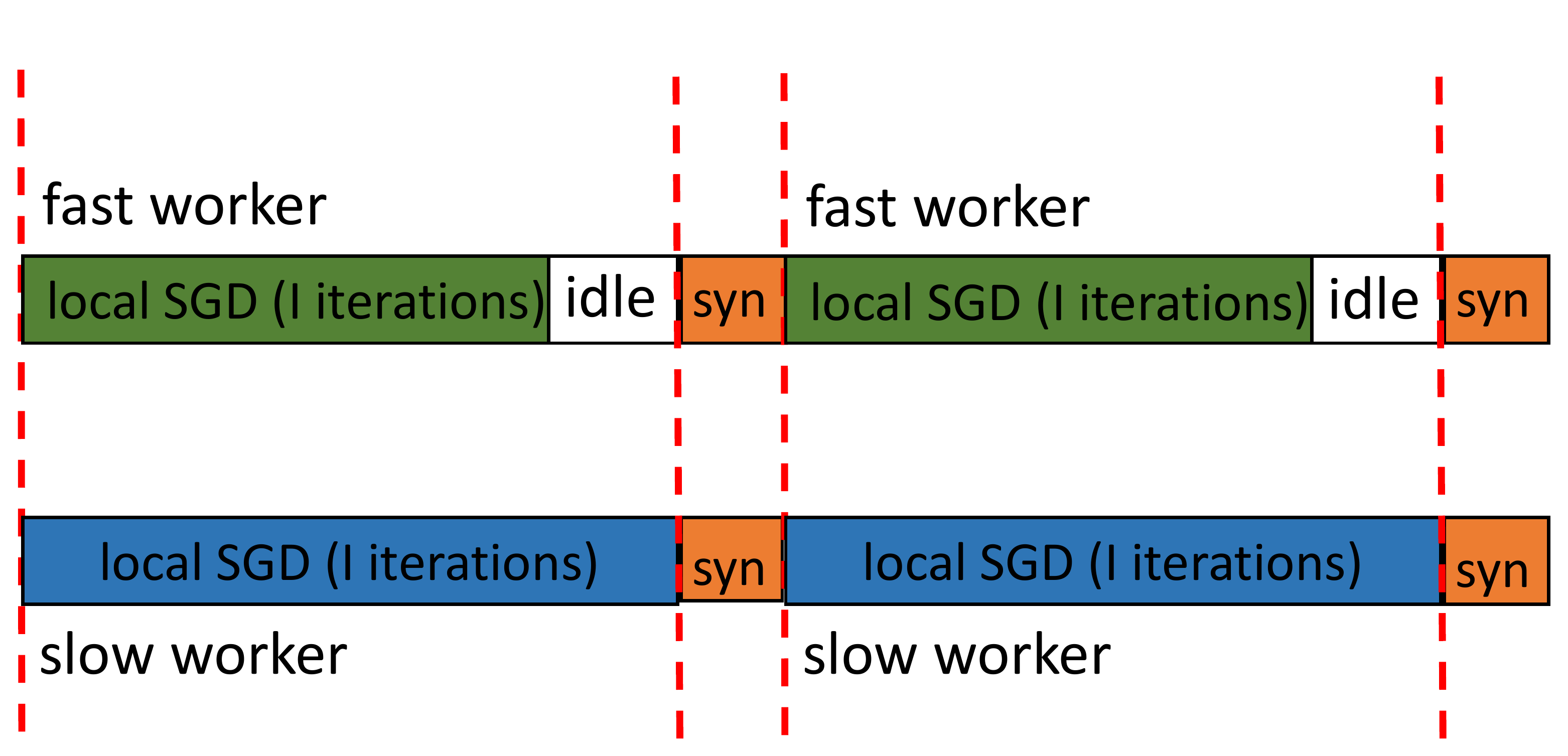}
\caption{An illustration of Algorithm \ref{alg:parallel-sgd} implemented in a $2$ worker heterogeneous network.   Orange ``syn'' rectangles represent the procedures to compute the node average.}\label{fig:1}
\end{figure}

Intuitively, if one worker finishes its $I$ iteration local SGD earlier, to avoid wasting its computation capability, we might want to let this worker continue running its local SGD until all the other workers finish their $I$ iteration local SGD. However, such a method can drag the node average too far towards the local solution at the fastest worker. Note that if $f_{i}(\cdot)$ in \eqref{eq:problem-form} are significantly different from each other such that the minimizer of $f_{i}(\cdot)$ at the $i$-th  worker, which is the fastest one, deviates the true minimizer of problem \eqref{eq:problem-form} too much, then dragging the node average towards the fastest worker's local solution is undesired.  In this subsection, we further assume that problem \eqref{eq:problem-form} satisfies the following assumption:

\begin{Assumption} \label{ass:identical-dist}
The distributions $\mathcal{D}_{i}$ in the definition of  each $f_{i}(\mathbf{x}) \overset{\Delta}{=}\mathbb{E}_{\zeta_{i}\sim \mathcal{D}_{i}} [ F_{i}(\mathbf{x}; \zeta_{i})]$ in \eqref{eq:problem-form} are identical.
\end{Assumption}

Note that Assumption \ref{ass:identical-dist} is satisfied if all local workers can access a common training data set or each local training data set is obtained from uniform sampling from the global training set.  Consider the restarted local SGD for heterogeneous networks described in Algorithm \ref{alg:parallel-sgd-heterogeneous}. Note that if $I_{i} \equiv I, \forall i$ for some fixed constant $I$, then Algorithm \ref{alg:parallel-sgd-heterogeneous} degrades to Algorithm \ref{alg:parallel-sgd}. 

\begin{algorithm}
\caption{PR-SGD in Heterogeneous Networks }\label{alg:parallel-sgd-heterogeneous}
\begin{algorithmic}[1]
\State {\bf Input:}  Set learning rate $\gamma > 0$ and epoch length of each worker $i$ as $I_{i}$.
\State {\bf Initialize:} Initialize $\mathbf{x}_i^{0,I_{i}}  = \overline{\mathbf{x}}^{0} \in \mathbb{R}^m$.
\For{epoch index $s=1$~\text{to}~$S$}
\State  Initialize $\mathbf{x}_{i}^{s,0} = \frac{1}{N}\sum_{i=1}^{N} \mathbf{x}_{i}^{s-1,I_{i}}$ as the node average of local worker solutions from the last epoch.
\State  Each worker $i$ in parallel runs its local SGD for $I_i$ iterations via:
\begin{align}
\mathbf{x}_{i}^{s,k} =\mathbf{x}_{i}^{s,k-1} - \gamma \mathbf{G}_{i}^{s,k},\quad \forall i \label{eq:parallel-sgd-heterogeneous-x-update}
\end{align}
where $\mathbf{G}_{i}^{s,k}$ is an unbiased stochastic gradient at point $\mathbf{x}_{i}^{s,k-1}$.
\EndFor
\end{algorithmic}
\end{algorithm}

 In practice, if the hardware configurations or measurements (from previous experiments) of each local worker are known, we can predetermine the value of each $I_{i}$, i.e., if worker $i$ is two times faster than worker $j$, then $I_{i} = 2 I_{j}$. Alternatively, under a more practical implementation, we can set a fixed time duration for each epoch and let each local worker keep running its local SGD until the given time elapses. By doing so, within the same time duration, the faster a worker is, the more SGD iterations it runs. In contrast, if we apply Algorithm \ref{alg:parallel-sgd} in this setting, then all local workers have to run the same number of SGD iterations as that can be run by the slowest worker within the given time interval.  This subsection shows that, under Assumption \ref{ass:identical-dist}, Algorithm \ref{alg:parallel-sgd-heterogeneous} can achieve a better performance than Algorithm \ref{alg:parallel-sgd} in heterogeneous networks where some workers are much faster than others.

Without loss of generality, this subsection always indexes local workers in a decreasing order of speed. That is, worker $1$ is the fastest while worker $N$ is the slowest. If we run Algorithm \ref{alg:parallel-sgd-heterogeneous} by specifying a fixed wall clock time duration for each epoch, during which each local worker keeps running its local SGD, then we have  $I_{1} \geq I_{2} \geq \cdots  \geq I_{N}$. Fix epoch index $s$, note that for $i\neq 1$, variables $\mathbf{x}_{i}^{s,k}$ with $k > I_{i}$ is never used. However, for the convenience of analysis, we define
\begin{align*}
\mathbf{x}_{i}^{s,k} \overset{\Delta}{=} \mathbf{x}_{i}^{s,k-1}, \forall i\neq 1, \forall k\in\{I_{i}+1, \ldots, I_{1}\}.
\end{align*}
Conceptually, the above equation can be interpreted as assuming worker $i$, which is slower than worker $1$, runs extra $I_{1} - I_{i}$ iterations of SGD by using $0$ as an imaginary stochastic gradient (with no computation cost).  See Figure \ref{fig:2} for a $2$ worker example where $I_{1} = 16$ and $I_{2} = 8$.  Using the definition $\overline{\mathbf{x}}^{s,k} \overset{\Delta}{=}  \frac{1}{N}\sum_{i=1}^{N} \mathbf{x}_{i}^{s,k}$, we have
{\small
\begin{align*}
\overline{\mathbf{x}}^{s,k} = \overline{\mathbf{x}}^{s,k-1} + \gamma \frac{1}{N} \sum_{i: I_{i} \geq k} \mathbf{G}_{i}^{s,k}, \forall s, \forall k\in\{1,2,\ldots, I_{1}\}.
\end{align*}
}

\begin{figure}[ht!] 
\centering
\includegraphics[width=0.48\textwidth]{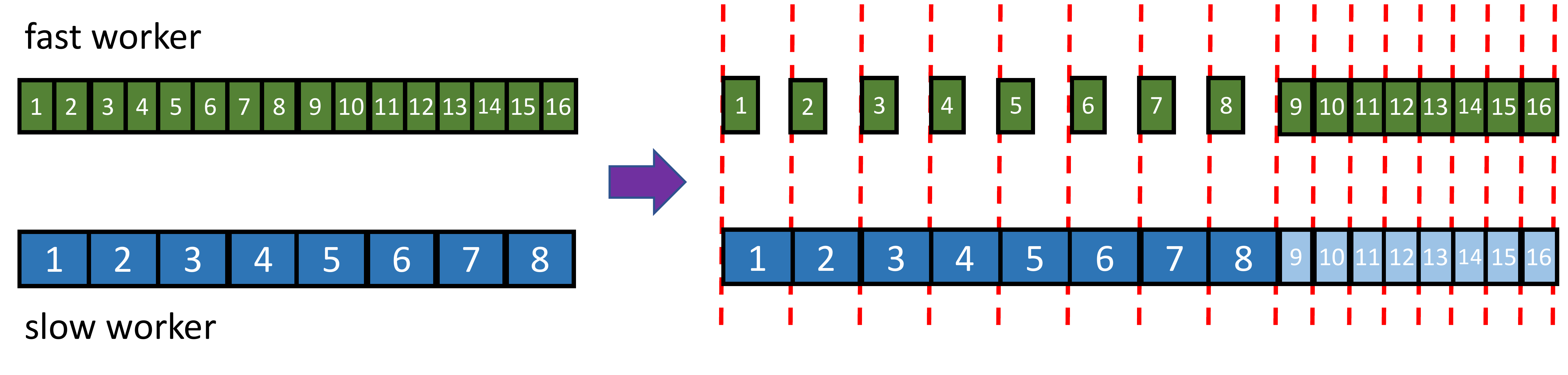}
\caption{{\bf Left}: A typical epoch of Algorithm \ref{alg:parallel-sgd-heterogeneous} in a heterogeneous network with $2$ workers. A wider rectangle means the SGD iteration takes a longer wall clock time.  {\bf Right}: Imagined extra SGD iterations with a $0$ stochastic gradient (in light blue rectangles) are added for the slow worker.}\label{fig:2}
\end{figure}

\begin{Thm} \label{thm:rate-heterogeneous} Consider problem \eqref{eq:problem-form} under Assumptions \ref{ass:basic} and \ref{ass:identical-dist}. Suppose all workers are indexed in a decreasing order of their speed, i.e., worker $1$ is the fastest and worker $N$ is the slowest. If $0 < \gamma \leq \frac{1}{L}$ in Algorithm \ref{alg:parallel-sgd-heterogeneous}, then for all $S\geq 1$, 
{\small
\begin{align}
&\frac{1}{S\frac{1}{N}\sum_{i=1}^{N} I_{i}}\sum_{s=1}^{S}\sum_{k=1}^{I_{1}}\frac{j_{k}}{N} \mathbb{E}[\Vert \nabla f(\overline{\mathbf{x}}^{s,k-1})\Vert^{2}] \nonumber \\
\leq& \frac{2}{\gamma S\frac{1}{N}\sum_{i=1}^{N} I_{i}} (f(\overline{\mathbf{x}}^{0}) -  f^{\ast})  + 4\gamma^{2} I_{1}^{2} G^{2}L^{2}+ \frac{L}{N} \gamma \sigma^{2} \label{eq:thm-rate-heterogeneous}
\end{align}
}%
where $j_{k}$ for each given $k$ is the largest integer in $\{1,2,\ldots, N\}$ such that $k \leq I_{j_{k}}$(That is, for each fixed $k$,  $j_{k}$ is the number of workers that are still using sampled true stochastic gradients to update their local solutions at iteration $k$.); and $f^{\ast}$ is the minimum value of problem \eqref{eq:problem-form}. 
\end{Thm}
\begin{proof}
See Supplement.
\end{proof}

The next corollary shows that Algorithm \ref{alg:parallel-sgd-heterogeneous} in heterogeneous networks can ensure the convergence and preserve the same $O(1/\sqrt{NT})$ convergence rate with the same $O(\frac{T^{1/4}}{N^{3/4}})$ communication reduction.
 
\begin{Cor}\label{cor:rate-heterogeneous}
Consider problem \eqref{eq:problem-form} under Assumptions \ref{ass:basic} and \ref{ass:identical-dist}. Let $T\geq N$. If we use $\gamma = \Theta(\frac{\sqrt{N}}{\sqrt{T}})$ such that $\gamma \leq \frac{1}{L}$, $I_i = \Theta\big(\frac{T^{1/4}}{N^{3/4}}\big), \forall i$ and $S = \frac{T}{I_N}$ in Algorithm \ref{alg:parallel-sgd-heterogeneous}, then 
\begin{align*}
\frac{1}{S\frac{1}{N}\sum_{i=1}^{N} I_{i}}\sum_{s=1}^{S}\sum_{k=1}^{I_{1}}\frac{j_{k}}{N} \mathbb{E}\left[\Vert \nabla f(\overline{\mathbf{x}}^{s,k-1})\Vert^{2}\right ] \leq O(\frac{1}{\sqrt{NT}}) 
\end{align*}
where $j_{k}$ for each given $k$ is the largest integer in $\{1,2,\ldots, N\}$ such that $k \leq I_{j_{k}}$. 
\end{Cor}
\begin{proof}
This simply follows by substituting values of $\gamma, I_i, S$ into \eqref{eq:thm-rate-heterogeneous} in Theorem \ref{thm:rate-heterogeneous}.
\end{proof}

\begin{Rem}
Note that once $I_{i}$ values are known, then $j_{k}$ for any $k$ in Theorem \ref{thm:rate-heterogeneous} and Corollary \ref{cor:rate-heterogeneous} are also available by its definition. To appreciate the implication of Theorem \ref{thm:rate-time-varying}, we recall that Algorithm \ref{alg:parallel-sgd} can be interpreted as a special case of Algorithm \ref{alg:parallel-sgd-heterogeneous} with $I_{i} \equiv I_{N}, \forall i$, i.e., all workers can only run the same number (determined by the slowest worker) of SGD iterations  in each epoch. In this perspective, Theorem \ref{thm:rate} (with $I = I_N$) implies that the performance of Algorithm \ref{alg:parallel-sgd} is given by 
{\small
\begin{align}
&\frac{1}{S I_N}\sum_{s=1}^{S}\sum_{k=1}^{I_{N}} \mathbb{E}[\Vert \nabla f(\overline{\mathbf{x}}^{s,k-1})\Vert^{2}] \nonumber \\\leq& \frac{2}{\gamma S I_N} (f(\overline{\mathbf{x}}^{0}) -  f^{\ast} )  + 4\gamma^{2} I_{N}^{2} G^{2}L^{2}+ \frac{L}{N} \gamma \sigma^{2} \label{eq:syn-rate-remark}
\end{align}
}%
Note that the left sides of \eqref{eq:thm-rate-heterogeneous} and \eqref{eq:syn-rate-remark} (weighted average expressions) can be attained by taking each $\overline{\mathbf{x}}^{s,k-1}$ randomly with a probability equal to the normalized weight in the summation.  The first error term in \eqref{eq:thm-rate-heterogeneous} is strictly smaller than that in \eqref{eq:syn-rate-remark} while the second error term in \eqref{eq:thm-rate-heterogeneous} is larger than that in \eqref{eq:syn-rate-remark}. Note that the constant factor $f(\overline{\mathbf{x}}^{0}) -  f^{\ast} $ in the first error term in \eqref{eq:syn-rate-remark} is large when a poor initial point $\overline{\mathbf{x}}^0$ is chosen (and dominates the second error term if $f(\overline{\mathbf{x}}^0) - f^\ast) \geq 2 \gamma^3 I_N^3 G^2 L^2 S$).  So the main message of Theorem \ref{thm:rate-heterogeneous} is that if a poor initial point is selected,  Algorithm \ref{alg:parallel-sgd-time-varying} can possibly converges faster than Algorithm \ref{alg:parallel-sgd} (at least for the first few epochs) in a heterogeneous network. 
\end{Rem}

\section{Experiment}
The superior training speed-up performance of model averaging has been  empirically observed in various deep learning scenarios, e.g., CNN for MNIST in \cite{Zhang16ArXiv}\cite{Kamp18ArXiv}\cite{McMahan17AISTATS}; VGG for CIFAR10 in \cite{Zhou17ArXiv}; DNN-GMM for speech recognition in \cite{Chen16ICASSP} \cite{Su18ArXiv}; and LSTM for language modeling in \cite{McMahan17AISTATS}. A thorough empirical study of ResNet over CIFAR and ImageNet is also available in the recent work \cite{Lin18ArXiv}. In Figures \ref{fig:3} and \ref{fig:4}, we compare model averaging, i.e., PR-SGD (Algorithm \ref{alg:parallel-sgd}) with $I\in\{4, 8, 16, 32\}$) with the classical parallel mini-batch SGD\footnote{The classical parallel mini-batch SGD is equivalent to Algorithm \ref{alg:parallel-sgd} with $I=1$. Our implementation with Horovod uses the more efficient ``all-reduce``method  rather than the ``parameter server" method to synchronize information between workers.} by training ResNet20 with CIFAR10 on a machine with $8$ P100 GPUs. Our implementation uses Horovod \cite{Sergeev18Horovod} for inter-worker communication and uses PyTorch 0.4 for algorithm implementations. 

\begin{figure}[ht!] 
	\centering
	\includegraphics[width=0.44\textwidth]{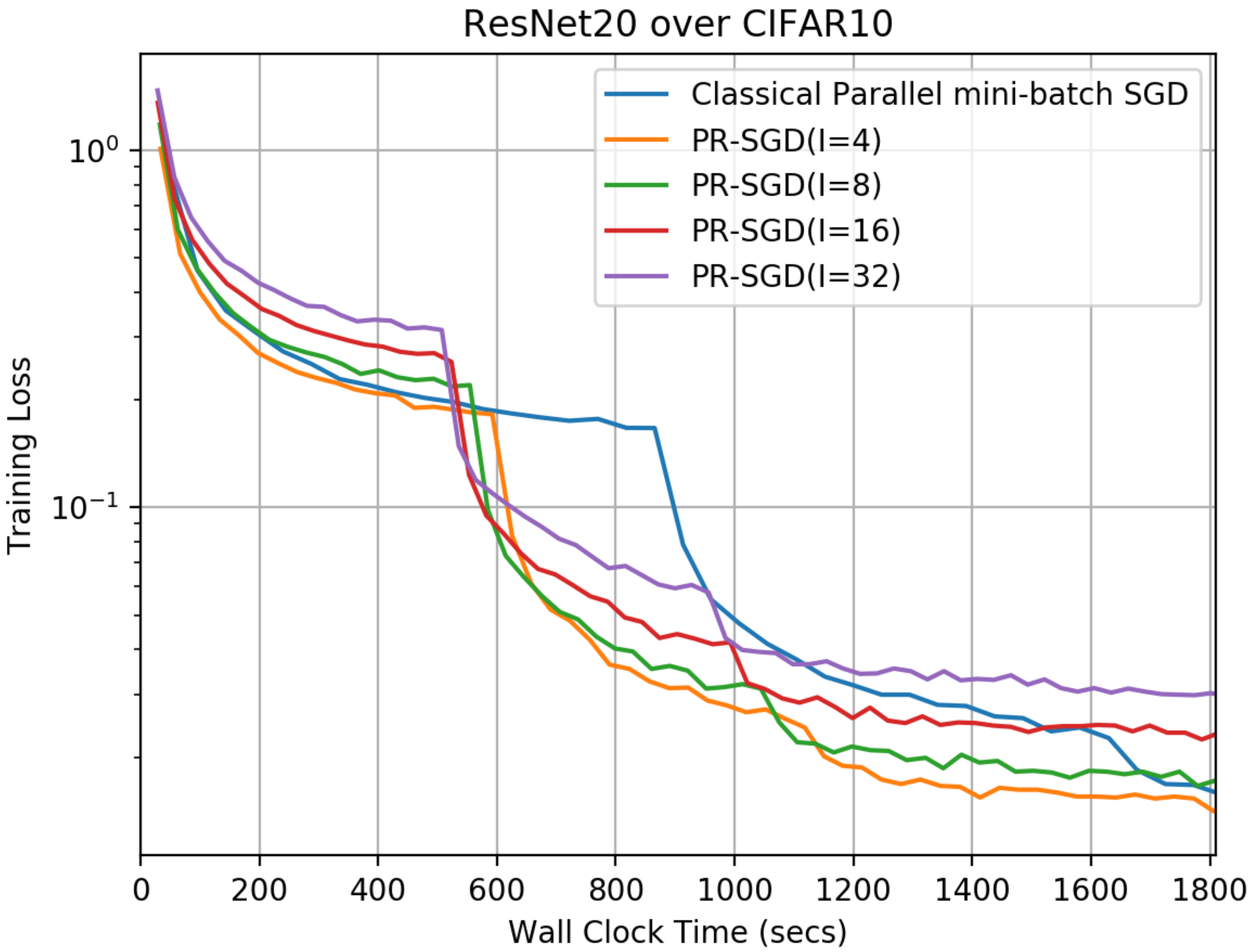}
	\caption{Training loss of ResNet20 over CIFAR10 on a machine with $8$ P100 GPUs. In all schemes, each worker uses a local batch size $32$ and momentum $0.9$. The initial learning at each worker is $0.1$ and is divided by $10$ when $8$ workers together access $150$ epochs and $275$ epochs of training data.}\label{fig:3}
\end{figure}

\begin{figure}[ht!] 
	\centering
	\includegraphics[width=0.44\textwidth]{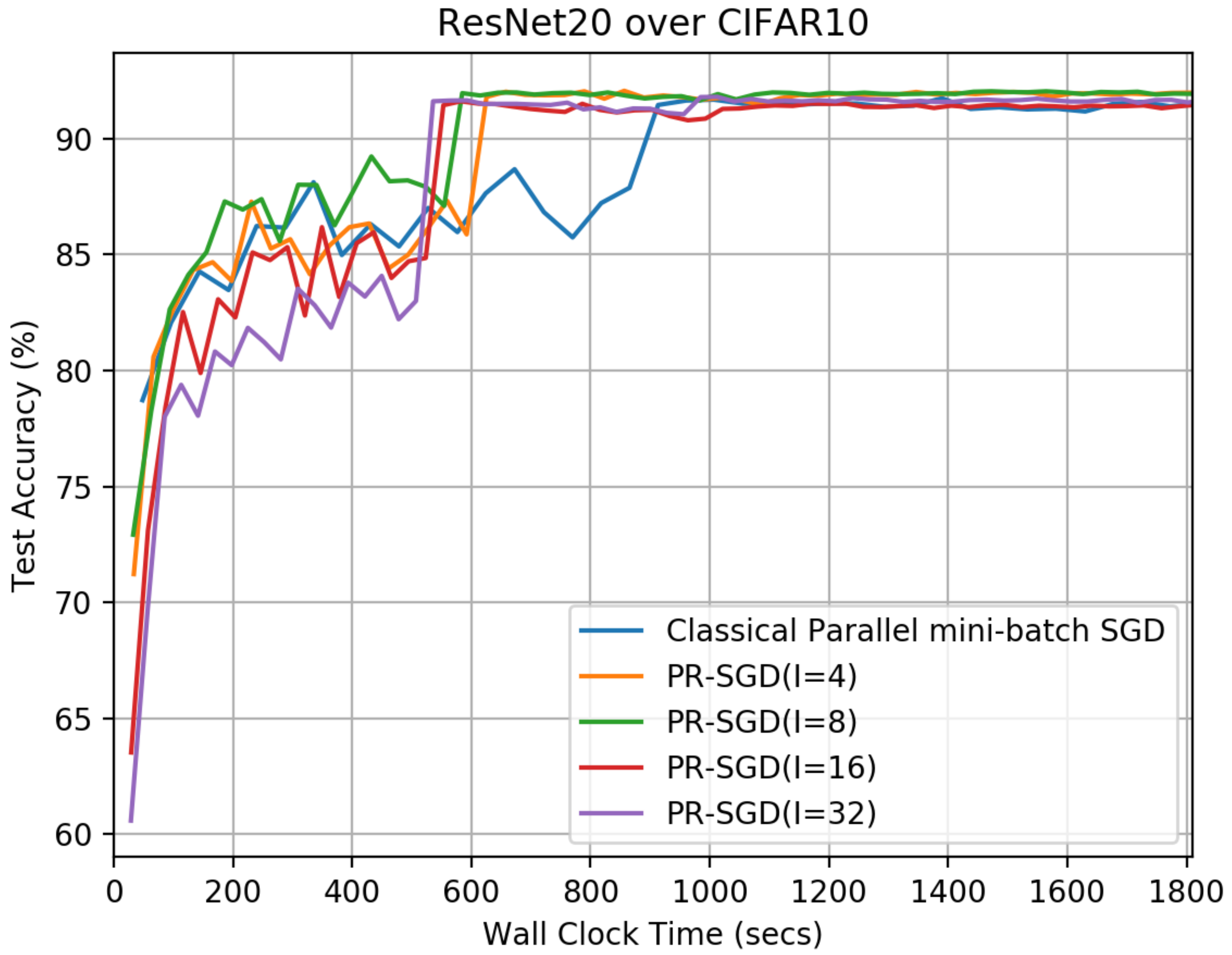}
	\caption{Test accuracy of ResNet20 over CIFAR10 on a machine with $8$ P100 GPUs. In all schemes, each worker uses a local batch size $32$ and  momentum $0.9$. The initial learning at each worker is $0.1$ and is divided by $10$ when $8$ workers together access $150$ epochs and $275$ epochs of training data.}\label{fig:4}
\end{figure}

\section{Conclusion}
This paper studies parallel restarted SGD, which is a theoretical abstraction of the  ``model averaging" practice widely used in training deep neural networks. This paper shows that parallel restarted SGD can achieve $O(1/\sqrt{NT})$ convergence for non-convex optimization with a number of communication rounds reduced by a factor $O(T^{1/4}$) compared with that required by the classical parallel mini-batch SGD.

\bibliography{mybibfile}
\bibliographystyle{aaai}
 
\newpage

\setcounter{page}{1}
\section{Supplement}

\subsection{Proof of Theorem \ref{thm:rate-time-varying}} \label{app:time-varying}

The next lemma extends Lemma \ref{lm:diff-avg-per-node}.
\begin{Lem}
Under Assumption \ref{ass:basic}, Algorithm \ref{alg:parallel-sgd-time-varying} ensures
\begin{align*}
\mathbb{E} \Vert \overline{\mathbf{x}}^{s,k} - \mathbf{x}_{i}^{s,k}\Vert^{2} \leq 4(\gamma^{s} K^{s})^{2} G^{2}, \forall i, \forall s, \forall k\in\{1,2,\ldots, K^{s}\}
\end{align*}
where $G$ is the constant defined in Assumption \ref{ass:basic}.
\end{Lem}
\begin{proof}
This lemma trivially extends Lemma \ref{lm:diff-avg-per-node} by restricting our attentions to each particular epoch $s$.
\end{proof}

{\bf Main Proof of Theorem \ref{thm:rate-time-varying}:} Fix $s$ and $k$. Following the lines in the proof of Theorem \ref{thm:rate} until \eqref{eq:pf-thm-rate-eq6} (by replacing $\overline{\mathbf{x}}^{t}, \gamma$ with $\overline{\mathbf{x}}^{s,k}, \gamma^{s}$, respectively), we have
{\small
\begin{align*}
&\mathbb{E}\left[f(\overline{\mathbf{x}}^{s,k})\right]  \\
\leq &\mathbb{E}\left[f(\overline{\mathbf{x}}^{s,k-1})\right]  - \frac{\gamma^{s} - (\gamma^{s})^{2}L}{2} \mathbb{E}[\Vert \frac{1}{N} \sum_{i=1}^{N} \nabla f_{i} (\mathbf{x}_{i}^{s,k-1})\Vert^{2}] \nonumber \\ &- \frac{\gamma^{s}}{2} \mathbb{E} [\Vert \nabla f(\overline{\mathbf{x}}^{s,k-1})\Vert^{2}] + 2(\gamma^{s})^{3} (K^{s})^{2} G^{2} L^2 + \frac{L}{2N} (\gamma^{s})^{2} \sigma^{2}
\end{align*}
}
Rearranging terms yields
\begin{align*}
&\frac{\gamma^{s}}{2} \mathbb{E} [\Vert \nabla f(\overline{\mathbf{x}}^{s,k-1})\Vert^{2}] \\
\leq &- \frac{\gamma^{s} - (\gamma^{s})^{2}L}{2} \mathbb{E}[\Vert \frac{1}{N} \sum_{i=1}^{N} \nabla f_{i} (\mathbf{x}_{i}^{s,k-1})\Vert^{2}]  \\&+ \mathbb{E}\left[f(\overline{\mathbf{x}}^{s,k-1})\right] - \mathbb{E}[f(\overline{\mathbf{x}}^{s,k})] \\&+ 2(\gamma^{s})^{3} (K^{s})^{2} G^{2} L^2 + \frac{L}{2N} (\gamma^{s})^{2} \sigma^{2}
\end{align*}

Note that when $s$ is sufficiently large, $\gamma^{s} = \frac{N}{s^{2/3}}$ is eventually between $(0,\frac{1}{L}]$. That is, for large $s$, e.g., $s > \lceil (NL)^{3/2} \rceil$, we have 
\begin{align}
&\frac{\gamma^{s}}{2} \mathbb{E}[\Vert \nabla f(\overline{\mathbf{x}}^{s,k-1})\Vert^{2}] \nonumber\\
\leq & \mathbb{E}\left[f(\overline{\mathbf{x}}^{s,k-1})\right] - \mathbb{E}\left[f(\overline{\mathbf{x}}^{s,k})\right] + 2(\gamma^{s})^{3} (K^{s})^{2} G^{2} L^2 \nonumber\\ &+ \frac{L}{2N} (\gamma^{s})^{2} \sigma^{2} \label{eq:pf-thm-time-varying-eq1}
\end{align}

Recall that  $f(\cdot)$ has stochastic gradients with bounded second order moments by Assumption \ref{ass:basic}. There exists a constant $C$ such that for $s \leq  \lceil (NL)^{3/2} \rceil$, we have
\begin{align}
&\frac{\gamma^{s}}{2} \mathbb{E}[\Vert \nabla f(\overline{\mathbf{x}}^{s,k-1})\Vert^{2}] \nonumber\\
\leq & \frac{C}{\sum_{s=1}^{\lceil (NL)^{3/2} \rceil} K^{s}} + \mathbb{E}[f(\overline{\mathbf{x}}^{s,k-1})] - \mathbb{E}[f(\overline{\mathbf{x}}^{s,k})] \nonumber \\ &+ 2(\gamma^{s})^{3} (K^{s})^{2} G^{2} L^2 + \frac{L}{2N} (\gamma^{s})^{2} \sigma^{2} \label{eq:pf-thm-time-varying-eq2}
\end{align}

Summing \eqref{eq:pf-thm-time-varying-eq2} over $s\in\{1,\ldots, \lceil (NL)^{3/2} \rceil \}, k\in\{1,2,\ldots, K^{s}\}$ and \eqref{eq:pf-thm-time-varying-eq1} over $s\in\{\lceil (NL)^{3/2} \rceil+1,\ldots, S \}, k\in\{1,2,\ldots, K^{s}\}$ (noting that $\overline{\mathbf{x}}^{s,T_{s}} = \overline{\mathbf{x}}^{s+1,0},\forall s$ and $\overline{\mathbf{x}}^{1,0} = \overline{\mathbf{x}}^0$) yields 
{\small
\begin{align}
&\sum_{s=1}^{S} \sum_{k=1}^{K^{s}} \frac{\gamma^{s}}{2}\mathbb{E}\left [\Vert \nabla f(\overline{\mathbf{x}}^{s,k-1})\Vert^{2}\right]  \nonumber \\
\leq & \mathbb{E}[f(\overline{\mathbf{x}}^{1,0})] - \mathbb{E}[f(\overline{\mathbf{x}}^{S,K^{S}})] + C + 2G^{2}L^{2} \sum_{s=1}^{S} (K^{s})^{2} \sum_{k=1}^{K^{s}} (\gamma^{s})^{3} \nonumber \\ &+ \frac{L\sigma^{2}}{2N} \sum_{s=1}^{S} \sum_{k=1}^{K^{s}} (\gamma^{s})^{2}  \nonumber \\
\overset{(a)}{\leq} & f(\overline{\mathbf{x}}^{0}) - f^{\ast} + C + 2G^{2}L^{2} \sum_{s=1}^{S} (K^{s})^{2} \sum_{k=1}^{K^{s}} (\gamma^{s})^{3} \nonumber \\ &+ \frac{L\sigma^{2}}{2N} \sum_{s=1}^{S} \sum_{k=1}^{K^{s}} (\gamma^{s})^{2} \label{eq:pf-thm-time-varying-eq3}
\end{align}
}%
where (a) follows because $f^{\ast}$ is the minimum value of problem \eqref{eq:problem-form}.

Note that there exists constant $c_{1} > 0$ such that  
\begin{align}
\sum_{s=1}^{S} (K^{s})^{2} \sum_{k=1}^{K^{s}} (\gamma^{s})^{3} =& \sum_{s=1}^{S} (\lceil \frac{s^{1/3}}{N} \rceil )^{3} (\frac{N}{s^{2/3}})^3 \nonumber \\
\leq & c_{1} \sum_{s=1}^{S} \frac{1}{s} \nonumber \\
=& O(\log(S)) \nonumber \\
\overset{(a)}{=}& O(\log(NT)) \label{eq:pf-thm-time-varying-eq4}
\end{align}
and there exits constant $c_{2} > 0$ such that 
\begin{align}
\sum_{s=1}^{S} \sum_{k=1}^{K^{s}} (\gamma^{s})^{2} =& \sum_{s=1}^{S} K^{s} (\gamma^{s})^{2} \nonumber \\
=& \sum_{s=1}^{S} \lceil \frac{s^{1/3}}{N} \rceil (\frac{N}{s^{2/3}})^2 \nonumber \\
\leq & c_{2} N \sum_{s=1}^{S} \frac{1}{s} \nonumber \\
= &O(N \log(S)) \nonumber \\
\overset{(a)}{=}& O(N\log(NT)) \label{eq:pf-thm-time-varying-eq5}
\end{align}
and there exits constant $c_{3} > 0$ such that
\begin{align}
\sum_{s=1}^{S} \sum_{k=1}^{K^{s}} \gamma^{s}=& \sum_{s=1}^{S} K^{s} \gamma^{s} \nonumber \\
=& \sum_{s=1}^{S} \lceil \frac{s^{1/3}}{N} \rceil  \frac{N}{s^{2/3}} \nonumber \\
\geq & c_{3}  \sum_{s=1}^{S}  \frac{s^{1/3}}{N}   \frac{N}{s^{2/3}} \nonumber \\
\geq & \Omega( S^{2/3}) \nonumber \\
\overset{(a)}{=} & \Omega(\sqrt{NT}) \label{eq:pf-thm-time-varying-eq6}
\end{align}
where (a) in the above three equations  \eqref{eq:pf-thm-time-varying-eq4}-\eqref{eq:pf-thm-time-varying-eq6} follows by noting that if $\sum_{s=1}^{S} K^{s} = \sum_{s=1}^{S} \lceil \frac{s^{1/3}}{N} \rceil = T$, then $S = \Theta( (NT)^{3/4} )$

Dividing both sides of \eqref{eq:pf-thm-time-varying-eq3} by $\sum_{s=1}^{S} \sum_{k=1}^{K^{s}} \frac{\gamma^{s}}{2}$ and substituting \eqref{eq:pf-thm-time-varying-eq4}-\eqref{eq:pf-thm-time-varying-eq6} into it yields
{\small
	\begin{align*}
\frac{1}{\sum_{s=1}^{S}\sum_{k=1}^{K^{s}}{\gamma^{s}}}  \sum_{s=1}^{S} \sum_{k=1}^{K^{s}} \mathbb{E}[ \gamma^s \Vert \nabla f(\overline{\mathbf{x}}^{s, k-1})\Vert^{2}] \leq O(\frac{\log(NT)}{\sqrt{NT}})
\end{align*}
}

\subsection{Proof of Theorem \ref{thm:rate-heterogeneous}} \label{app:rate-heterogeneous}

\begin{Lem} \label{lm:diff-avg-per-node-heterogeneous}
Under Assumption \ref{ass:basic}, Algorithm \ref{alg:parallel-sgd-heterogeneous} ensures
\begin{align*}
\mathbb{E} \Vert \overline{\mathbf{x}}^{s,k} - \mathbf{x}_{i}^{s,k}\Vert^{2} \leq 4\gamma^{2} k^{2}G^{2} , \forall i, \forall s, \forall k\in\{1,2,\ldots, I_{i}\} 
\end{align*}
where $G$ is the constant defined in Assumption \ref{ass:basic}.
\end{Lem}
\begin{proof}
Fix $i, s$ and $k$. Let $\overline{\mathbf{y}} = \frac{1}{N}\sum_{i=1}^{N} \mathbf{x}_{i}^{s-1, I_{i}}$, which is the common initial point of epoch $s$. Note that
{\small 
\begin{align*}
\mathbf{x}_{i}^{s,k} = \overline{\mathbf{y}} - \gamma \sum_{\tau=1}^{k} \mathbf{G}_{i}^{s,\tau},
\end{align*}
}%
and
{\small 
\begin{align*}
\overline{\mathbf{x}}^{s,k} = \overline{\mathbf{y}} - \gamma \sum_{\tau=1}^{k} \frac{1}{N} \sum_{l: \tau \leq I_{l}} \mathbf{G}_{l}^{s,\tau}.
\end{align*}
}

Thus, we have
{\small
\begin{align*}
&\mathbb{E}\Vert \overline{\mathbf{x}}^{s,k} - \mathbf{x}_{i}^{s,k}\Vert^{2} \\
= & \gamma^{2} \mathbb{E} \Vert \sum_{\tau=1}^{k} \mathbf{G}_{i}^{s,\tau}-\sum_{\tau=1}^{k} \frac{1}{N} \sum_{l: \tau \leq I_{l}} \mathbf{G}_{l}^{s,\tau} \Vert^{2} \\
\overset{(a)}{\leq}&  2 k \gamma^{2} \Big(\sum_{\tau=1}^{k} \mathbb{E}\Vert \mathbf{G}_{i}^{s,\tau}\Vert^{2} + \sum_{\tau=1}^{k} \mathbb{E}\Vert \frac{1}{N} \sum_{l: \tau \leq I_{l}} \mathbf{G}_{l}^{s,\tau} \Vert^{2}\Big) \\
\overset{(b)}{\leq} &2k \gamma^{2} \Big( kG^{2} + kG^{2} \Big)
\end{align*}
}
where (a) follows by using the inequality $\Vert \sum_{i=1}^{n} \mathbf{z}_{i}\Vert^{2} \leq n \sum_{i=1}^{n} \Vert \mathbf{z}_{i}\Vert^{2}$ for any vectors $\mathbf{z}_{i}$ and any positive integer $n$; and (b) follows by using the same inequality (noting that there are less than $N$ terms in the summation $\sum_{l: \tau \leq I_{l}} \mathbf{G}_{l}^{s,\tau}$) and applying Assumption \ref{ass:basic}.
\end{proof}

{\bf Main Proof of Theorem \ref{thm:rate-heterogeneous}:}  Fix $s$ and $k\in\{1,2,\ldots,I_{1}\}$. Recall that $j_{k}$ is the largest integer in $\{1,2,\ldots, N\}$ such that $k \leq I_{j_{k}}$. That is, at iteration $k$ in epoch $s$, only the first $j_{k}$ workers are using true stochastic gradients to update and the other workers stop updating (since they are too slow and do not have the chance to perform the $k$-th update in this epoch).  Note that as $k$ increases to $I_{1}$, $j_{k}$ decreases to $1$. By the definition of $j_{k}$, we have 
\begin{align*}
\overline{\mathbf{x}}^{s,k} = \overline{\mathbf{x}}^{s,k-1} + \gamma \frac{1}{N} \sum_{i=1}^{j_{k}} \mathbf{G}_{i}^{s,k}.
\end{align*}

By the smoothness of $f$ , we have 
{\small
\begin{align}
\mathbb{E}[f(\overline{\mathbf{x}}^{s,k})] \leq&  \mathbb{E}[f(\overline{\mathbf{x}}^{s,k-1})] +  \mathbb{E}[\langle \nabla f(\overline{\mathbf{x}}^{s,k-1}), \overline{\mathbf{x}}^{s,k} - \overline{\mathbf{x}}^{s,k-1}\rangle] \nonumber \\&+ \frac{L}{2} \mathbb{E}[\Vert \overline{\mathbf{x}}^{s,k} - \overline{\mathbf{x}}^{s,k-1}\Vert^{2}]  \label{eq:pf-thm-heterogeneous-eq1}
\end{align}
}
Similarly to \eqref{eq:pf-thm-rate-eq2}, we can show
{\small
\begin{align}
 &\mathbb{E}[\Vert \overline{\mathbf{x}}^{s,k} - \overline{\mathbf{x}}^{s,k-1}\Vert^{2}] \nonumber\\
 \leq & \frac{j_{k}}{N^{2}} \gamma^{2} \sigma^{2}  +  \frac{\gamma^{2} j_{k}^{2}}{N^{2}} \mathbb{E} [ \Vert  \frac{1}{j_{k}}\sum_{i=1}^{j_{k}} \nabla f_{i}(\mathbf{x}_{i}^{s,k-1})\Vert^{2}]   \label{eq:pf-thm-heterogeneous-eq2}
\end{align}
}
We further have
{\small
\begin{align}
&\mathbb{E}[\langle \nabla f(\overline{\mathbf{x}}^{s,k-1}), \overline{\mathbf{x}}^{s,k} - \overline{\mathbf{x}}^{s,k-1}\rangle] \nonumber\\
\overset{(a)}{=}& -\frac{\gamma j_{k}}{N} \mathbb{E} [\langle \nabla f(\overline{\mathbf{x}}^{s,k-1}),  \frac{1}{j_{k}}\sum_{i=1}^{j_{k}} \mathbf{G}_{i}^{s,k}\rangle ] \nonumber \\
\overset{(b)}{=}& -\frac{\gamma j_{k}}{N} \mathbb{E} [\langle \nabla f(\overline{\mathbf{x}}^{s,k-1}),  \frac{1}{j_{k}} \sum_{i=1}^{j_{k}} \nabla f_{i} (\mathbf{x}_{i}^{s,k-1})\rangle ] \nonumber \\
\overset{(c)}{=}& - \frac{\gamma j_{k} }{2N} \mathbb{E} \Big[ \Vert \nabla f(\overline{\mathbf{x}}^{s,k-1})\Vert^{2} + \Vert \frac{1}{j_{k}} \sum_{i=1}^{j_{k}} \nabla f_{i} (\mathbf{x}_{i}^{s,k-1})\Vert^{2} \nonumber \\ &\qquad\qquad - \Vert \nabla f(\overline{\mathbf{x}}^{s,k-1}) - \frac{1}{j_{k}} \sum_{i=1}^{j_{k}} \nabla f_{i} (\mathbf{x}_{i}^{s,k-1})\Vert^{2}\Big]   \label{eq:pf-thm-heterogeneous-eq3}
\end{align}
}%
where (a)-(c) follows by using the same arguments used in showing \eqref{eq:pf-thm-rate-eq3}.

Substituting \eqref{eq:pf-thm-heterogeneous-eq2} and \eqref{eq:pf-thm-heterogeneous-eq3} into \eqref{eq:pf-thm-heterogeneous-eq1} yields
{\small
\begin{align}
&\mathbb{E}\left[f(\overline{\mathbf{x}}^{s,k})\right] \nonumber \\
\leq &\mathbb{E}\left[f(\overline{\mathbf{x}}^{s,k-1})\right] - \frac{\gamma j_{k}}{2N} (1 - \frac{j_{k} L \gamma}{N}) \mathbb{E}[\Vert \frac{1}{j_{k}} \sum_{i=1}^{j_{k}} \nabla f_{i} (\mathbf{x}_{i}^{s,k-1})\Vert^{2}] \nonumber \\ &- \frac{\gamma j_{k}}{2N} \mathbb{E}[\Vert \nabla f(\overline{\mathbf{x}}^{s,k-1})\Vert^{2}]  \nonumber\\ &+ \frac{\gamma j_{k}}{2N} \mathbb{E}[
\Vert \nabla f(\overline{\mathbf{x}}^{s,k-1}) - \frac{1}{j_{k}} \sum_{i=1}^{j_{k}} \nabla f_{i} (\mathbf{x}_{i}^{s,k-1}) \Vert^{2} ] + \frac{j_{k}L}{2N^{2}} \gamma^{2} \sigma^{2} \nonumber \\
\overset{(a)}{\leq} & \mathbb{E}[f(\overline{\mathbf{x}}^{s,k-1})]  - \frac{\gamma j_{k}}{2N} \mathbb{E}[\Vert \nabla f(\overline{\mathbf{x}}^{s,k-1})\Vert^{2} ]  \nonumber\\ &+ \frac{\gamma j_{k}}{2N} \mathbb{E}[
\Vert \nabla f(\overline{\mathbf{x}}^{s,k-1}) - \frac{1}{j_{k}} \sum_{i=1}^{j_{k}} \nabla f_{i} (\mathbf{x}_{i}^{s,k-1}) \Vert^{2}] + \frac{j_{k}L}{2N^{2}} \gamma^{2} \sigma^{2}   \label{eq:pf-thm-heterogeneous-eq4}
\end{align}
}
where (a) follows because $0 < \gamma \leq \frac{1}{L} \leq \frac{N}{j_{k} L}$ (noting that $j_{k} \leq N, \forall k$).

Note that
{\small 
\begin{align}
&\mathbb{E}[\Vert \nabla f(\overline{\mathbf{x}}^{s,k-1}) - \frac{1}{j_{k}} \sum_{i=1}^{j_{k}} \nabla f_{i} (\mathbf{x}_{i}^{s,k-1})\Vert^{2}] \nonumber \\
\overset{(a)}{=}& \mathbb{E}[\Vert \frac{1}{j_{k}} \sum_{i=1}^{j_{k}}\nabla f_{i}(\overline{\mathbf{x}}^{s,k-1}) - \frac{1}{j_{k}} \sum_{i=1}^{j_{k}} \nabla f_{i} (\mathbf{x}_{i}^{s,k-1})\Vert^{2}] \nonumber \\
=& \frac{1}{j_{k}^{2}} \mathbb{E}[\Vert \sum_{i=1}^{j_{k}} \big( \nabla f_{i}(\overline{\mathbf{x}}^{s,k-1}) - \nabla f_{i} (\mathbf{x}_{i}^{s,k-1}) \big)\Vert^{2}] \nonumber \\
\overset{(b)}{\leq}& \frac{1}{j_{k}} \mathbb{E} [ \sum_{i=1}^{j_{k}} \Vert \nabla f_{i}(\overline{\mathbf{x}}^{s,k-1}) - \nabla f_{i} (\mathbf{x}_{i}^{s,k-1})\Vert ^{2}] \nonumber 
\end{align}
\begin{align}
\overset{(c)}{\leq}& L^2\frac{1}{j_{k}} \sum_{i=1}^{j_{k}}\mathbb{E}[ \Vert \overline{\mathbf{x}}^{s,k-1} - \mathbf{x}_{i}^{s,k-1}\Vert^{2}] \nonumber \\
\overset{(d)}{\leq}&  4\gamma^{2} I_1^{2} G^{2} L^2 \label{eq:pf-thm-heterogeneous-eq5}
\end{align}
}%
where (a) follows because $f_{i}(\cdot)$ are identical by Assumption \ref{ass:identical-dist}; 
(b) follows by using the inequality $\Vert \sum_{i=1}^{n} \mathbf{z}_{i}\Vert^{2} \leq n \sum_{i=1}^{n} \Vert \mathbf{z}_{i}\Vert^{2}$ for any vectors $\mathbf{z}_{i}$ and any integer $n$; (c) follows from the smoothness of each $f_{i}$ by Assumption \ref{ass:basic}; and (d) follows from Lemma \ref{lm:diff-avg-per-node-heterogeneous}.

Substituting \eqref{eq:pf-thm-heterogeneous-eq5} into \eqref{eq:pf-thm-heterogeneous-eq4} yields
{\small
\begin{align}
\mathbb{E}\left[f(\overline{\mathbf{x}}^{s,k})\right] \leq &\mathbb{E}\left[f(\overline{\mathbf{x}}^{s,k-1})\right]  - \frac{\gamma j_{k}}{2N} \mathbb{E}\left[\Vert \nabla f(\overline{\mathbf{x}}^{s,k-1})\Vert^{2}\right ] \nonumber\\& + \frac{2\gamma^{3} j_{k} I_{1}^{2} G^{2}L^{2}}{N}  + \frac{j_{k}L}{2N^{2}} \gamma^{2} \sigma^{2}  
\end{align}
}

Dividing both sides by $\frac{\gamma}{2}$ and rearranging terms yields
{\small
\begin{align}
&\frac{j_{k}}{N} \mathbb{E}\left[\Vert \nabla f(\overline{\mathbf{x}}^{s,k-1})\Vert^{2}\right ] \nonumber\\
 \leq &\frac{2}{\gamma} \big( \mathbb{E}\left[f(\overline{\mathbf{x}}^{s,k-1})\right]  -  \mathbb{E}[f(\overline{\mathbf{x}}^{s,k})]  \big) + \frac{4j_{k}}{N} \gamma^{2} I_{1}^{2} G^{2}L^{2}+ \frac{j_{k}}{N}\frac{L}{N} \gamma \sigma^{2}  \label{eq:pf-thm-heterogeneous-eq6}
\end{align}
}%

As a sanity check of our analysis, we note that if $I_{1} = I_{2} = \cdots = I_{N}$, then $j_{k} = N, \forall k$ and \eqref{eq:pf-thm-heterogeneous-eq6} reduces to \eqref{eq:pf-thm-rate-eq7}. (Recall that Algorithm \ref{alg:parallel-sgd-time-varying} is identical to Algorithm \ref{alg:parallel-sgd} in this case.)

Summing \eqref{eq:pf-thm-heterogeneous-eq6} over $k\in\{1,\ldots, I_{1}\}$, $s\in\{1,2,\ldots, S\}$ (noting that $\overline{\mathbf{x}}^{s,I_{1}} = \overline{\mathbf{x}}^{s+1,0}, \forall s$ and $\overline{\mathbf{x}}^{1,0} = \overline{\mathbf{x}}^{0}$) and dividing both sides by $\sum_{s=1}^{S} \sum_{k=1}^{I_{1}} \frac{j_{k}}{N} = S \frac{1}{N}\sum_{i=1}^{N} I_{i}$ yields
{\small
\begin{align}
&\frac{1}{S\frac{1}{N}\sum_{i=1}^{N} I_{i}}\sum_{s=1}^{S}\sum_{k=1}^{I_{1}}\frac{j_{k}}{N} \mathbb{E}[\Vert \nabla f(\overline{\mathbf{x}}^{s,k-1})\Vert^{2}] \nonumber \\ \leq &\frac{2}{\gamma S\frac{1}{N}\sum_{i=1}^{N} I_{i}} \big( \mathbb{E}[f(\overline{\mathbf{x}}^{0})]  -  \mathbb{E}[f(\overline{\mathbf{x}}^{S,I_{1}})]  \big)  + 4\gamma^{2} I_{1}^{2} G^{2}L^{2}+ \frac{L}{N} \gamma \sigma^{2} \nonumber \\
\overset{(a)}{\leq}&\frac{2}{\gamma S\frac{1}{N}\sum_{i=1}^{N} I_{i}} \Big(f(\overline{\mathbf{x}}^{0}) -  f^{\ast} \Big)  + 4\gamma^{2} I_{1}^{2} G^{2}L^{2}+ \frac{L}{N} \gamma \sigma^{2} \label{eq:pf-thm-heterogeneous-eq7}
\end{align}
}%
where (a) follows because $f^{\ast}$ is the minimum value of problem \eqref{eq:problem-form}.

\end{document}